\documentclass[a4paper,10pt]{amsart}
\usepackage{amsmath}
\usepackage{mathrsfs}
\usepackage{amsfonts}
\usepackage{amssymb}
\usepackage{euscript,eufrak,verbatim}
\usepackage[dvips]{graphicx}
\usepackage[final]{epsfig}
\def\cC{\mathcal{C}}

\newtheorem{theorem}{Theorem}[section]
\newtheorem{lemma}[theorem]{Lemma}
\newtheorem{proposition}[theorem]{Proposition}

\DeclareMathSymbol{\varnothing}{\mathord}{AMSb}{"3F}

\begin{document}

\title{Absolutely Continuous Invariant Measures of Piecewise Linear
Lorenz Maps}
\author{Yi Ming DING, Ai Hua FAN and Jing Hu YU}

\medskip
\date{}
\maketitle

\pagestyle{myheadings}

\markboth{ Acims of piecewise linear Lorenz maps}{\sc Yi-Ming Ding
Ai-Hua Fan and Jing-Hu Yu}

\begin{abstract}
Consider  piecewise  linear Lorenz maps on $[0,\ 1]$ of the following form
\[ f_{a,b,c}(x)=\left \{ \begin{array}{ll}
ax+1-ac & x \in [0,  c) \\
b(x-c) & x \in (c,  1].
\end{array}
\right.
\]
We prove that $f_{a,b,c}$ admits an absolutely continuous invariant
probability measure (acim) $\mu$ with respect to the Lebesgue
measure if and only if $f_{a,b,c}(0) \le f_{a,b,c}(1)$, i.e.
$ac+(1-c)b \ge 1$. The acim is unique and ergodic unless $f_{a,b,c}$
is conjugate to a rational rotation. The equivalence between the
acim and the Lebesgue measure is also fully investigated via the
renormalization theory.
\end{abstract}




\vspace{2em}

\section{ Introduction }
\medskip

Lorenz maps are one-dimensional maps with a single singularity,
which  arise  as  Poincar\'e return maps for flows on branched
manifolds that model the strange attractors of Lorenz systems. A
{\em Lorenz map} on the interval $I:=[0,1]$ is a map $f:I \to I$
such that for some critical point
$c\in (0,1)$ we have\\
\indent
(i)   $f$ is continuous and strictly increasing on $[0,c)$ and on $(c,1]$;\\
\indent
(ii)  $\lim_{x \uparrow c}f(x)=1$, $\lim_{x \downarrow c}f(x)=0$.\\
A Lorenz map $f$ is said to be {\em piecewise linear} if it is
linear on both intervals $[0, c)$ and $(c, 1]$. Such a map is of the
form
\begin{equation}\label{LinearLorenzMap0}
f_{a,b,c}(x)= \left \{ \begin{array}{ll}
ax+1-ac & x \in [0,  c) \\
b(x-c) & x \in (c,  1],
\end{array}
\right.
\end{equation}
where $a>0, b>0$, $0<c<1$, $ac\le 1$ and $b(1-c)\le 1$.

Let $\beta>1$. The map $T_{\beta}(x)=\beta x \ \ (\mod  1)$ is the
well  known $\beta$-shift related to $\beta$-expansion (\cite{R}).
Assume $0\le \alpha <1$. The transformation $T_{\beta, \alpha}$
defined by
$$
T_{\beta, \alpha}(x)=\beta x + \alpha \ \ \ (\!\!\mod  1)
$$
is a natural generalization of  $\beta$-shift. There are many works
done  on $T_{\beta,\alpha}$ (see \cite{G, Go1, Ho, Pal, P1, P, P3,
W}). When $1<\beta \le 2$, $T_{\beta, \alpha}$ is a piecewise linear
Lorenz map. In fact, $T_{\beta, \alpha}=f_{\beta, \beta, c}$ with
$c=(1-\alpha)/\beta$. Recently, Dajani et al \cite{DHK} studied
another variation $S_{\beta,\alpha}$ of $\beta$-shift. For $0 <
\alpha < 1$ and $1 < \beta < 2$,
\begin{equation}\label{LinearLorenzMap10} S_{\beta,\alpha}(x)=\left \{ \begin{array}{ll}
\beta x & x \in [0,\  1/\beta) \\
\alpha (x-1/\beta) & x \in (1/\beta,\  1],
\end{array}
\right.
\end{equation}
which is the piecewise linear Lorenz map  $f_{\beta,
\alpha,1/\beta}$.

Lorenz maps arise as return maps to a cross-section of a semi-flow
on a two dimensional branched manifold (cf. \cite{ABS}, \cite{GW},
\cite{Wi}). The flow lines starting from $c$ never return to $I$. So
usually the map is considered not defined at $c$ (cf. \cite{GS}).
But it is also convenient to regard $c$ as two points $c+$ and $c-$,
the right and left  of $c$, so that the Lorenz map is a continuous
map defined on the disconnected compact space $[0,\ c-] \bigcup \
[c+,\ 1]$. Different dynamical aspects of Lorenz maps are studied in
the literatures such as rotation interval, asymptotic periodicity,
topological entropy and renormalization etc (see \cite{ALMT},
\cite{Dy}, \cite{DF1}, \cite{GS}). In this paper we shall study the
absolutely continuous invariant probability measures (acim for
short) of piecewise linear Lorenz maps. The existence of acim and
the equivalence of acim with respect to the Lebesgue measure are
studied.


The Lebesque measure is clearly quasi invariant under $f_{a,b,c}$.
Let $P_{a,b,c}$ be the associated Perron-Frobenius operator and let
$$ A_n (h)=\frac{1}{n}\sum_{i=0}^{n-1}P_{a,b,c}^ih,\ \ \ \ \ \ h\in
L^1([0,1]).$$ Our results are stated in the following two theorems.

\medskip

\noindent {\bf Theorem A.} \ {\it The piecewise linear Lorenz map
$f_{a,b,c}$ admits an absolutely continuous  invariant probability
measure $\mu$ with respect to the Lebesgue measure if and only if
$f_{a,b,c}(0) \le f_{a,b,c}(1)$, i.e. $ac+b(1-c) \ge 1$. More
precisely,

\begin{enumerate}

\item If $ac+b(1-c)=1$ and $ \log a / \log b $ is rational, then
there exists positive integer $n$ such that $f^n_{a,b,c}(x)=x$ for
all $x \in [0, 1]$. Consequently, for each density $g$ on $[0,\ 1]$,
$A_n(g)$ is the density of an invariant measure of $f_{a,b,c}$.

\item If $ac+b(1-c)=1$ and $ \log a / \log b $ is irrational,
then the acim is unique and its density is bounded from below and
from above by the two constants ${\displaystyle
\left(\frac{a}{b}\right)^4}$ and
 ${\displaystyle \left(\frac{a}{b}\right)^{-4}}$.

\item If $ac+b(1-c)>1$, then the acim is unique and its density is
of bounded variation.

\item If $ac+b(1-c)<1$, then
$f_{a,b,c}$ admits no acim.

\end{enumerate}

}
\medskip
It is well known to Lasota and Yorke (\cite{LY1}) that a strongly
expanding interval map $f$ (i.e. $|f'(x)|>\lambda>1$ except finite
points) admits an acim with respect to the Lebesgue measure. It is
also known that a piecewise linear Lorenz map with a fixed point
also admits an acim with respect to the Lebesgue measure (cf.
\cite{DHK,D}). These results don't apply to the Lorenz maps defined
by (\ref{LinearLorenzMap0}) which, in general, are not strongly
expanding and admit no fixed point.

Suppose that $f_{a,b,c}$ admits a unique acim  with respect to the
Lebesgue measure.  If $f_{a,b,c}$ is a homeomorphism (i.e., $ac
+b(1-c)=1$) with irrational rotation number, we shall see from the
proof of Theorem A  that the acim of $f_{a,b,c}$ is equivalent to
the Lebesgue measure. If $ac+b(1-c)>1$, the acim is not necessarily
equivalent to the Lebesgue measure, even if $f_{a,b,c}$ is strongly
expanding (i.e.$a>1$ and $b>1$).  For example, Parry \cite{P3}
proved that the acim of symmetric piecewise linear Lorenz map
$f_{a,a,1/2}$ is not equivalent to the Lebesgue measure if and only
if $1<a<\sqrt{2}$. We point out that the support of the acim of
$T_{\beta,\alpha}$ was studied in \cite{G, Ho}.

Assume that $ac+b(1-c)>1$. As we shall see in Lemma
\ref{transitive},  the acim of $f_{a,b,c}$ is equivalent to the
Lebesgue measure if and only if $f_{a,b,c}$ is transitive, i.e.
$\bigcup_{n \ge 0}f_{a,b,c}^n(U)$ is dense in $I$ for each non-empty
open set $U\subset I$. In general, the transitivity of a Lorenz map
is not easy to check. Palmer \cite{Pal} studied the transitivity of
$T_{\beta,\alpha}$ by using so-called primary cycle (see also
\cite{G}). Alves et al introduced a topological invariant to study
the transitivity of $T_{\beta,\alpha}$ (\cite{AFS1}). The conditions
of both primary cycle and the topological invariant of Alves et al
are difficult to check too. We will provide a rather simple
criterion of the transitivity for the piecewise linear Lorenz maps
$f_{a,b,c}$ with $ac+b(1-c)>1$.

Let us describe our criterion. Assume $ac+b(1-c)>1$. Then
$f_{a,b,c}$ admits periodic points, because it admits positive
topological entropy (\cite{ALMT}). Let $\kappa$ be the minimal
period of the periodic points of $f_{a,b,c}$. Assume $2\le
\kappa<\infty$. Then $f_{a,b,c}$ admits a unique $\kappa-$periodic
orbit. Let $P_L$ and $P_R$ be adjacent $\kappa-$periodic points such
that $c \in [P_L,\ P_R]$. It can be proved that $f_{a,b,c}^\kappa$
is linear on $[P_L,\ c)$ and on $(c,\ P_R]$ (Lemma \ref{minimal
period}). Write
$$A:=f^{\kappa}_{a,b,c}(c+),\ \ \ \ \ \ \ B:=f^{\kappa}_{a,b,c}(c-).$$
That $\kappa=2$ means $f_{a,b,c}$ has no fixed point,
$f_{a,b,c}(0)=A<B=f_{a,b,c}(1)$ and $c \in [A,\ B]$. When
$\kappa=2$, let
 $$M:=\min\left\{\frac{c-A}{B-c}, \frac{B-c}{c-A} \right\}.$$

\medskip
\noindent {\bf Theorem B.} \ {\it Suppose $ac + b(1-c)>1$.  If
$\kappa=1$, then the acim of $f_{a,b,c}$ is  equivalent to the
Lebesgue measure. If $\  2\le \kappa<\infty$, then the acim of
$f_{a,b,c}$ is equivalent to the Lebesgue measure if and only if
$$[A,\ B]\backslash [P_L, \ P_R]\neq \emptyset \ \ \ \  or\ \ \ \
[A,\ B]=[P_L, \ P_R].$$ In particular, when $\kappa=2$, the acim of
$f_{a,b,c}$ is
 equivalent to the Lebesgue measure if and only if
  \[\left \{ \begin{array}{ll}
ab>1+M & \mbox{\rm if}\ \ \ M<1 \\
ab\ge 2 & \mbox{\rm if}\ \ \  M=1.
\end{array}
\right.
\]
 }

\medskip
See Figure $1$ for a piecewise linear map whose  acim is not
equivalent to the Lebesgue measure.

\begin{figure}
\centering
\includegraphics[width=0.8\textwidth]{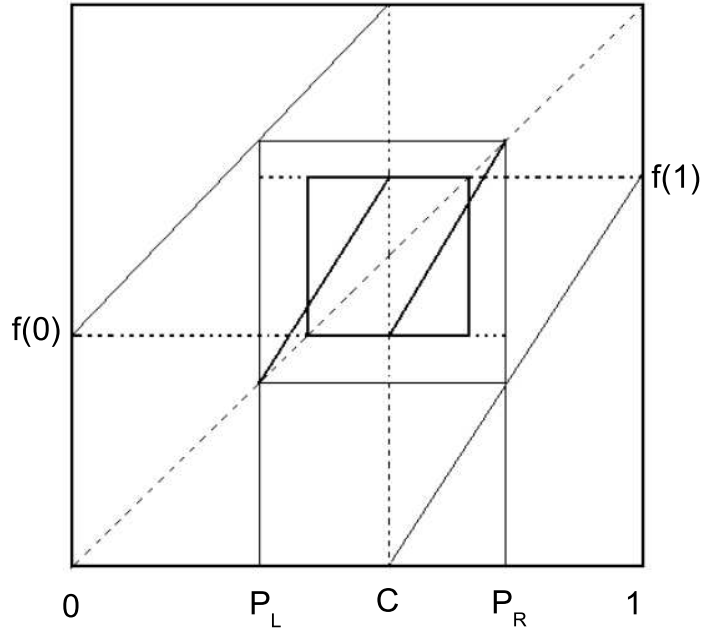}
\caption{A piecewise linear Lorenz map with $\kappa=2$, and
$[f^2(c+),\ f^2(c-)]=[f(0),\ f(1)] \subseteq (P_L,\ P_R)$, whose
acim is not equivalent to the Lebesgue measure.}
\end{figure}

Parry \cite{P3}  proved that the acim of $f_{a,a,1/2}$ ($1<a \le 2$)
is not equivalent to the Lebesgue measure if and only if $1<a < \sqrt{2}$. This may be obtained  as a special case of
Theorem B.




We shall collect some basic useful facts in $\S 2$, including
rotation number, Lyapunov exponent, Frobenius-Perron operator and
renormalization. Theorem A is proved in $\S 3$ and Theorem B in $\S
4$. Densities of some piecewise linear Lorenz maps will be presented
in $\S 5$.

\setcounter{equation}{0}

\section{ Preliminaries }

In this section, we present some facts concerning the rotation
number, Lyapunov exponent, Frobenius-Perron operator and the
renormalization of Lorenz maps, which will be useful later.

\subsection{ Rotation number and  Lyapunov exponent}\ \ \ \

 We denote by $e:
\mathbb{R} \to \mathbb{S}^1=\{z\in \mathbb{C} :|z|=1\}$ the natural
covering map $e(x)=\exp(2\pi ix)$. Let $f$ be a Lorenz map, not
necessarily linear. There exists a map $F: \mathbb{R} \to
\mathbb{R}$ such that $e\circ F=f \circ e$ and $F(x+1)=F(x)+1$. $F$
is called a degree one lifting of $f$. Furthermore, if $F(0)=f(0)$,
then there exists a unique such lifting (cf. \cite{ALMT}).

The {\em rotation number} of $f$ at $x$ is defined by
$$
\rho(x)= \limsup_{n \to \infty} \frac{F^n(x)-x}{n}.
$$
It is known that the set of all rotation numbers $\rho(x)$ of $f$ is
an interval and that this interval is reduced to a singleton when
$f(0)=f(1)$ (\cite{Ito}). The rotation number is tightly relate to
the number of returns of $x$ into the interval $(c,\ 1]$, defined by
\begin{equation}\label{Returnnumber}
m_n(x)= \# \left\{0\le i <n: f^i(x) \in (c,\ 1] \right\}
\end{equation}

\begin{lemma}  \label{rotation-frequence}
Let $f$ be a  Lorenz map (not necessarily piecewise linear). Let $F$
be the unique degree one lifting map of $f$ such that $F(0) = f(0)$.
Then for any $n \ge 1$ and any $x \in [0,\ 1]$ we have
$$
F^n(x)=m_n(x)+f^n(x).
$$

\end{lemma}

\proof The proof of this Lemma would be found in the literatures, we
give a proof for completeness. We prove it by induction. First note
that
\[f(x)= \left \{ \begin{array}{ll}
F(x) & x \in [0,  c) \\
F(x)-1 & x \in (c,  1]
\end{array}
\right.
\]
which implies the desired equality for $n=1$:
$$
F(x)={\bf 1}_{(c,\ 1]}(x)+f(x).
$$
Suppose now that the equality is true for an arbitrary $n$. Since
$F(x+k)=F(x)+k$ for all positive integers $k$, by the hypothesis of
induction we have
$$
F^{n+1}(x)=F(F^{n}(x))= \sum_{i=0}^{n-1}{\bf 1}_{(c,\ 1]}(f^i(x))+
F(f^n(x)).
$$
According to what we have proved for $n=1$, we get $F(f^n(x)) = {\bf
1}_{[c, 1)}(f^n(x)) + f^{n+1}(x)$ so that
$$
F^{n+1}(x)=\sum_{i=0}^{n}{\bf 1}_{(c,\
1]}(f^i(x))+f^{n+1}(x)=m_{n+1}(x)+f^{n+1}(x).
$$
$\hfill\Box$

Write
$$
C_f=\bigcup_{n \ge 0}f^{-n}(c).
$$
When $f=f_{a,b,c}$ we write $C_{a,b,c}:=C_{f_{a,b,c}}$.  The {\em
Lyapunov exponent} of $f$ at $x \notin C_f$ is defined by
 $$
 \lambda(f,x)=\limsup_{n \to
\infty}\frac{1}{n} \log (f^n)'(x).
$$
For piecewise linear Lorenz map $f_{a,b,c}$,
\begin{equation}\label{RotationLiapunov}
\frac{d f_{a,b,c}^n(x)}{d x} =  a^{n- m_n(x)} b^{m_n(x)}, \quad
\forall x \in C_{a,b,c}.
\end{equation}

For any linear Lorenz map such that $f_{a,b,c}(0) = f_{a,b,c}(1)$,
its rotation number and its Lyapunov exponent are determined in the
following way.


\begin{lemma}  \label{Rotation-Lyapunov}Let $f_{a,b,c}$ be a piecewise
linear Lorenz map such that $f_{a,b,c}(0)=f_{a,b,c}(1)$. Let $\rho$
be the rotation number of $f_{a,b,c}$. Then we have
 \begin{equation}\label{uniform}
 \lim_{n\to \infty}\sup_{x \in [0,1]}
\left|\frac{F^n(x)-x}{n}-\rho\right|=0.
\end{equation}
The rotational number $\rho$ is the solution of the equation
$$
a^{1-\rho}b^{\rho}=1.
$$
The number $\rho$ is rational if and only if $ \log a / \log b $ is
rational.
 Furthermore, we have
$$\lambda(f_{a,b,c},x)=0 \qquad \qquad (\forall x \notin
C_{a,b,c}).
$$
\end{lemma}

\proof   The uniform convergence follows from the observation
$$F^n(0)<F^n(x)<F^{n+1}(0) \qquad (\forall x \in [0,\ 1])$$
and the fact that ${\displaystyle \rho = \lim_{n \to \infty}
\frac{F^n(0)}{n}}$.

According to  Lemma \ref{rotation-frequence}, the rotation number
$\rho$ of $f_{a,b,c}$ is nothing but the frequency of visits to
$(c,\ 1]$ of any given point of $[0,\ 1]$. Since $f_{a,b,c}$ is
piecewise linear, for $ x \notin C_{a,b,c}$ we have
\begin{equation}\label{derivative-of-Sn}
 (f_{a,b,c}^n)'(x)= \prod_{i=0}^{n-1}
f'_{a,b,c}(f_{a,b,c}^i(x))=a^{n-m_n(x)}b^{m_n(x)}
\end{equation} where
$ m_n(x) $ is defined by (\ref{Returnnumber}).
 It follows
that
\begin{eqnarray*}
 \lambda(f_{a,b,c},x)
 =  \lim_{n \to \infty} \frac{1}{n} \log
a^{n-m_n(x)}b^{m_n(x)}= \log a^{1-\rho}b^{\rho}.
 \end{eqnarray*}

Let $\lambda= \log a^{1-\rho}b^{\rho}$. We will prove
$a^{1-\rho}b^{\rho}=1$ by showing $\lambda =0$, which implies that
$$
\rho=1+\frac{\log b}{\log a - \log b}.
$$
So $\rho$ is rational if and only if $ \log a / \log b $ is
rational. Suppose $\lambda>0$. According to Lemma
\ref{rotation-frequence}, (\ref{uniform})   and
(\ref{derivative-of-Sn}),   there exists a positive integer $N$ such
that
$$
(f_{a,b,c}^N)'(x)>1  \qquad (\forall x \notin C_{a,b,c}).
$$
This and  the piecewise linearity of $f_{a,b,c}$ imply that
$f_{a,b,c}^N$ is piecewise expanding. However, it is not possible
because $f_{a,b,c}^N$ is a homeomorphism. Thus $\lambda \le 0$. In
the same way, one proves $\lambda \ge 0$  by considering
$f_{a,b,c}^{-1}$. Hence we get $\lambda=0$. $\hfill\Box$

\begin{lemma}  \label{Rational-Rotation}
Let $f_{a,b,c}$ be a piecewise linear Lorenz map such that
$f_{a,b,c}(0)=f_{a,b,c}(1)$. Let $\rho$ be the rotation number of
$f_{a,b,c}$. If $\rho$ is irrational, then
\begin{equation}\label{ReturnRotation}
|m_n(x)-n \rho| \le 4 \ \ \ \ \ (\forall x \in [0,\ 1],\  \forall
n\ge 1).
 \end{equation}
\end{lemma}

\proof Since $\rho$ is irrational, $f_{a,b,c}$ is topologically
conjugate to the rigid irrational rotation $R_{\rho}$, i.e.
$R_{\rho}(x)=\rho +x$ (cf. \cite{MS}, p. 38-39). In other words,
 there exists a continuous
strictly increasing function $\pi$ on $[0,\ 1]$ onto $[0,\ 1]$
such that
\begin{equation}\label{conjugate}
 \pi \circ f_{a,b,c} \circ \pi^{-1}=R_{\rho}.
\end{equation}
Let $F$, $G$ and $G^{-1}$ be the degree one lifting map of
$f_{a,b,c}$, $\pi$ and $\pi^{-1}$ respectively.  The lifted form
of (\ref{conjugate}) is
$$
G \circ F \circ G^{-1}(x)=x+ \rho.
$$
By induction, we have
\begin{equation}\label{R-conjugate}
G \circ F^n \circ G^{-1}(x)=x+ n \rho.
\end{equation}
Write
\begin{equation}\label{estimation}
\begin{array}{lcl}
|m_n(x)-n \rho| & \le & |m_n(x)-F^n(x)|+|F^n(x)-G \circ F^n(x)| \\
&  +&|G \circ F^n(x)-G \circ F^n \circ G^{-1}(x)|\\
&  +&|G \circ F^n \circ G^{-1}(x)-n \rho|.
\end{array}
\end{equation}
Now we estimate the four terms on the right hand side. Notice
first that $F$, $G$ and $G^{-1}$ are increasing functions on
$\mathbb{R}$ and that $F(x)-x$, $G(x)-x$ and $G^{-1}(x)-x$ are
$1$-periodic functions on $\mathbb{R}$ taking with values in
$[0,1]$. So when $|x-y| \le 1$ we have
\begin{equation}\label{FF-GG}
|F(x)-F(y)| \leq 1, \ \ \ |G(x)-G(y)| \leq 1, \ \ \
|G^{-1}(x)-G^{-1}(y)| \leq 1.
\end{equation}
According to Lemma \ref{rotation-frequence}, we have the following
estimate for the first term:
$$|m_n(x)-F^n(x)| \le 1.
$$
The fact $0\le G(x) -x\le 1$ implies immediately an estimate for
the second term:
$$
|F^n(x)-G \circ F^n(x)| \le 1.
$$
Repeating the first inequality in (\ref{FF-GG}) we get
$|F^n(x)-F^n(y)| \leq 1$ when $ |x-y|\le 1$. This and the fact
$|G^{-1}(x)-x| \le 1$ imply an estimate  for the third term:
$$|G
\circ F^n(x)-G \circ F^n \circ G^{-1}(x)| \le 1.
$$
A direct consequence of (\ref{R-conjugate}) is the following
estimate for the fourth term:
 $$|G \circ F^n
\circ G^{-1}(x)-n \rho| \le 1.$$  The estimation
(\ref{ReturnRotation}) is thus proved.
 $\hfill \Box$

\medskip

\subsection{Comparison of rotation numbers in different maps }

Let $f_{a,b,c}$ be a piecewise linear Lorenz map. If $a>1$ and
$b>1$, it is well known that $f_{a,b,c}$ is expanding and then
admits a unique acim. If $a<1$ and $b<1$,  $f_{a,b,c}$ is
contracting and then admits no acim. So we may assume that $a>1
\ge b$ or $b> 1 \ge a$. In these cases, we will compare
 $f_{a,b,c}$ with homeomorphic piecewise linear Lorenz maps $f_{a_0,b,c}$ and $f_{a,b_0,c}$,
 where $a_0=\frac{1-b(1-c)}{c}$ and $b_0=\frac{1-ac}{1-c}$. See Figure 2 for the
pictures of $f_{a_0,b,c}$ and $f_{a,b_0,c}$. More precisely, in the
case $a
> 1 \ge b$,  we compare $f_{a,b,c}$ with
\begin{equation} f_{a_0,b,c}(x)=\left \{ \begin{array}{ll}
a_0x+1-ac & x \in [0,\  c) \\
b(x-c) & x \in (c,\  1]
\end{array}
\right.\label{homeo1}
\end{equation}
if $f_{a,b,c}(0)<f_{a,b,c}(1)$, and compare $f_{a,b,c}$ with
\begin{equation}
\label{homeo2} f_{a,b_0,c}(x)=\left \{ \begin{array}{ll}
ax+1-ac & x \in [0,\  c) \\
b_0(x-c) & x \in (c,\  1]
\end{array}
\right.
\end{equation}
if $f_{a,b,c}(0)>f_{a,b,c}(1)$.

In the case $b>1 \ge a$, we compare $f_{a,b,c}$ with $f_{a,b_0,c}$
when $f_{a,b,c}(0)>f_{a,b,c}(1)$, and compare $f_{a,b,c}$ with
$f_{a_0,b,c}$ when $f_{a,b,c}(0)<f_{a,b,c}(1)$.

\begin{figure}
\centering
\includegraphics[width=0.90\textwidth]{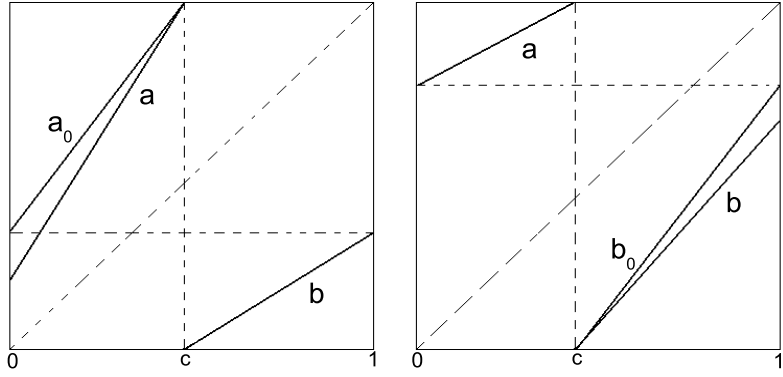}
\caption{The comparisons between piecewise linear Lorenz maps.}
\end{figure}

\begin{lemma} \label{compare}
Let $f_{a_0,b,c}$ and $f_{a,b_0,c}$ be defined as above. We have the
following conclusions:
\begin{enumerate}
\item If $f_{a,b,c}(0)<f_{a,b,c}(1)$, then for all $x \in [0, 1]$,
\[
\left \{ \begin{array}{ll} \rho(f_{a,b,c}, x) \le \rho(f_{a_0,b,c}) \ \ \ \ & \mbox{\rm if} \ \ \ \  a>1 \ge b \\
 \rho(f_{a,b,c}, x) \ge \rho(f_{a,b_0,c})\ \ \ \ & \mbox{\rm if} \ \ \ \
 a \le 1<b.
\end{array}
\right.
\]

\item If  $f_{a,b,c}(0)>f_{a,b,c}(1)$, then for all $x \in [0, 1]$,
\[
\left \{ \begin{array}{ll} \rho(f_{a,b,c}, x) \ge \rho(f_{a_0,b,c})\ \ \ \ &  \mbox{\rm if} \ \ \ \ a\ge1>b\\
 \rho(f_{a,b,c}, x) \le \rho(f_{a,b_0,c})\ \ \ \ &  \mbox{\rm if} \ \ \ \  a<1\le b.
\end{array}
\right.
\]

\end{enumerate}
\end{lemma}

\begin{proof} Let $F$ be the degree one lifting  of $f_{a,b,c}$ and $F_0$ be
the degree one lifting of $f_{a_0,b,c}$. For the case
$f_{a,b,c}(0)<f_{a,b,c}(1)$ and $a>1\ge b$, since $f_{a,b,c}(x) \leq
f_{a_0,b,c}(x)$ for all $ x \in [0,\ 1]$, we have $F(y) \leq F_0(y)$
for all $y \in \mathbb{R}$. It follows that $F^n(y) \leq F^n_0(y)$
for  $y \in \mathbb{R}$ and for every positive integer $n$.
Therefore $\rho(f_{a,b,c},x) \le \rho(f_{a,b_0,c})$. The other
inequalities can be similarly proved.
\end{proof}

\subsection{Frobenius-Perron operator and invariant density}

The Frobenius-Perron operator associated with $f_{a,b,c}$ is defined
as follows: for any $h \in L^1(I)$,
$$
P_{a,b,c} h(x)=\frac{{\bf 1}_{[1-ac,\
1]}(x)}{a}h\left(\frac{x-(1-ac)}{a}\right) +\frac{{\bf 1}_{[0,\
b(1-c)]}(x)}{b}h\left(\frac{x+bc}{b}\right).
$$
The invariant density $h_*$ of the Frobenius-Perron operator corresponds to an acim $\mu$ of $f_{a,b,c}$, $\mu(A)=\int_A h_* d m$, where $A \in \mathcal{B}$ is a Borel set and $m$ is the Lebesgue measure on $I$.

\begin{lemma}\label{theorem a} Let $P$ be the Frobenius-Perron operator associated with $f_{a,b,c}$. Then we have the following statements:

\begin{enumerate}

\item If there exists positive integer $n \ge 1$ such that $f_{a,b,c}^n(x)=x$ for all $x \in I$, then for each density $g$ on $[0,\ 1]$, $A_n(g)$ is an invariant density of $P_{a,b,c}$.

\item If there exists a constant $r>0$ such that for all positive integer $n$,
$$ rm(A)\le m(f^{-n}_{a,b,c}(A)) \le m(A)/r,\ \ \ \ \forall\  A \in
\mathcal{B},$$ then $f_{a,b,c}$ admits a unique acim whose density
is bounded by $r$ and $1/r$.

\item If there exists a positive integer $n$ such that
 $f_{a,b,c}^n$  is strongly expanding, i.e., $(f_{a,b,c}^n)'(x)>\lambda>1$ for all $x \in I$ except finite points, then $f_{a,b,c}$ admits an acim whose density is of bounded variation.

\item If there exists a positive integer $n$ such that $(f_{a,b,c}^n)'(x)<\lambda<1$ for all $x \in I$ except finite points, then $f_{a,b,c}$ admits no acim.
\end{enumerate}

\end{lemma}

\begin{proof}
The assertions (1) and (4) are obvious. The assertion (3) is a
direct consequence of Lasota and Yorke's Theorem (\cite{LM, LY1}).
Now we prove (2). The assumption in this case means that $r \le
P_{a,b,c}^n {\bf 1} \le 1/r$. So $\{A_n({\bf 1})\}_{n \ge 0}$ is
weakly precompact in $L^1(I)$. From the weakly compactness of
$\{A_n({\bf 1})\}_{n \ge 0}$ we can extract a subsequence
$A_{n_k}({\bf 1})$ that converges weakly to $g$ and $P_{a,b,c} g=g$.
By the abstract ergodic Theorem of Kakutani and Yosida (\cite{LM}),
$A_n({\bf 1})$ converges strongly to $g$. This implies that $g$ is
an invariant density of $P_{a,b,c}$ and $r \le g(x) \le 1/r$.
\end{proof}

In the third case, $f_{a,b,c}$ is said to be eventually piecewise
expanding  \cite{Go2}.

\subsection{Renormalization of Lorenz map}

Let $f_{a,b,c}$ be a piecewise linear Lorenz map satisfying
$ac+b(1-c)>1$. The equivalence between the acim of $f_{a,b,c}$ and
the Lebesgue measure is a question of transitivity of $f_{a,b,c}$
(see Lemma \ref{transitive}). One can describe the transitivity of
$f_{a,b,c}$ by using the device of renormalization.

A Lorenz map $f: I \to I$ is said to be {\it
renormalizable} if there is a proper subinterval $[u,\ v]$ which contains the critical point $c$, and
integers $\ell, r>1$ such that the map $g: [u,\ v] \to [u,\ v]$
defined by
\begin{equation} \label{renormalization}
g(x)=\left \{
\begin{array}{ll}
f^{\ell}(x) & x \in [u,\  c), \\
f^{r}(x) & x \in (c,\  v],
\end{array}
\right.
\end{equation}
is itself a Lorenz map on $[u,\ v]$.

A Lorenz map $f$ is said to be expanding if the preimages of the
critical point is dense in $I$.  The renormalization theory of
expanding Lorenz map is well understood (see \cite{Dy,GS}). The
transitivity of an expanding Lorenz map can be characterized by its
renormalization. For example, $f$ is transitive if it is not
renormalizable (\cite{Dy}).

Let $f$ be an expanding Lorenz map. The renormalizability of $f$ is
closely related to the periodic orbit with minimal period. Denote
$\kappa$ the smallest period of the periodic points of $f$. If
$\kappa=1$ (i.e., $f$ admits a fixed point), we must have $f(0)=0$
or $f(1)=1$ because  $f$ is expanding. It follows that $f$ is
transitive (\cite{Dy}). If $\kappa=\infty$, i.e. $f$ admits no
periodic point, then $f$ is topologically conjugates to an
irrational rotation on the circle because $f$ is expanding
(\cite{GS}). For the case $1<\kappa<\infty$, we have the following
Lemma.

\begin{lemma} (\cite{Dy}) \label{minimal period}
Let $f: [0,\ 1] \to [0,\ 1]$ be an expanding Lorenz map with
$1<\kappa<\infty$.
\begin{enumerate}
\item The minimal period of $f$ is equal to $\kappa=m+2$, where
 $$m=\min\{i \ge 0: f^{-i}(c) \in [f(0),\ f(1)]\}.$$

\item $f$ admits a unique
$\kappa$-periodic orbit  $O$.

\item  Let $P_L$ and $P_R$ be adjacent points in $O$ such that $c \in [P_L,\ P_R]$.
Then $f^{\kappa}$ is continuous on $[P_L,\ c)$ and on $(c,\ P_R]$.
Moreover,  we have
\begin{equation}
\label{equality}\bigcup_{i=0}^{\kappa-1}f^i([P_L,\
P_R])=I.\end{equation}
\end{enumerate}
\end{lemma}

For general expanding Lorenz map $f$, it is difficult to check wether
$f$ is renormalizable or not. However, for  piecewise linear Lorenz map $f_{a,b,c}$ satisfying $ac+b(1-c)>1$, one can check the
 renormalizability easily. According to the proof of Theorem A, $f_{a,b,c}$ is expanding. Denote $O$ as the $\kappa$-periodic orbit,
 and $$D:=\overline{\bigcup_{n \ge 0}f_{a,b,c}^{-n}(O)}.
 $$

\begin{lemma}(\cite{CD}) \label{dic} If $D \neq O$, then $f_{a,b,c}$ is not renormalizable.

\end{lemma}

\setcounter{equation}{0}

\section{ Existence of absolutely continuous invariant measure }

Now we  prove Theorem A by distinguishing four cases:
$f_{a,b,c}(0)=f_{a,b,c}(1)$ and $ \log a / \log b $ is rational,
$f_{a,b,c}(0)=f_{a,b,c}(1)$ and $ \log a / \log b $ is irrational,
$f_{a,b,c}(0)<f_{a,b,c}(1)$ and $f_{a,b,c}(0)>f_{a,b,c}(1)$.  In the
first case, we will show that some power of $f_{a,b,c}$ is identity,
i.e., there exists $n>0$ such that $f^{n}_{a,b,c}(x)=x$ for all
$x\in I$. In the second case we will prove  $$ rm(A)\le
m(f^{-n}_{a,b,c}(A)) \le m(A)/r,\ \ \ \ \forall\ A \in
\mathcal{B},$$ for some constant $r$ and $n \ge 0$. In the third
case  we will show that some power of $f_{a,b,c}$ is expanding. In
the forth case, we will compare $f_{a,b,c}$ with a suitable
homeomorphic piecewise linear Lorenz map and prove that some power
$f^n_{a,b,c}$ is contracting.

\subsection{Proof of Theorem A when $f_{a,b,c}(0)=f_{a,b,c}(1)$ and $ \log a / \log b $ is rational} \
\ \ \ \ \ \ \ \ \

According to Lemma \ref{theorem a}, it suffice to prove the
following proposition.

\begin{proposition}\label{case1} Suppose $f_{a,b,c}(0)=f_{a,b,c}(1)$ and $ \log a / \log b $ is rational. Then there exists positive integer $n$ such that $f_{a,b,c}^n(x)=x$ for all $x \in I$.
\end{proposition}

\begin{proof} In this case, $f_{a,b,c}$ can be regarded as a homeomorphism on the unit circle. Since $ \log a / \log b $ is rational, the rotation number of $f_{a,b,c}$ is also rational (Lemma \ref{Rotation-Lyapunov}). Write
 ${\displaystyle \rho(f_{a,b,c})=
\frac{m}{n}}$ with $(m,\ n)=1$. We shall prove that
$f^n_{a,b,c}(x)=x$ for all $x \in [0,\ 1]$.

Since ${\displaystyle\rho(f_{a,b,c})=\frac{m}{n}}$, $f_{a,b,c}$
admits an $n-$periodic orbit (\cite{ALMT}). Let $p_1<p_2<\cdots<p_n$
be an $n$-periodic orbit. The orbit forms a partition of $I$:
$$[p_1,\ p_2), \cdots,  [p_{n-1},\ p_n),\ [p_n, 1] \cup [0, p_1),$$ and
$f_{a,b,c}$ maps one subinterval onto the next one in the partition.
Each subinterval in the partition contains only one point in
$C_{a,b,c}$. Since $\rho(f_{a,b,c})= \frac{m}{n}$, $c \in [p_{n-m},
p_{n-m+1})$.

If $c$ doesn't belong to the periodic orbit,
$c \in (p_{n-m}, \ p_{n-m+1})$.  Consider the
interval $[p_{n-m}, c)$, it follows that $f_{a,b,c}^n$ is
continuous and linear on $[p_{n-m}, c)$ because $f_{a,b,c}^n$ has
only one discontinuity $c$ in $[p_{n-m}, p_{n-m+1}]$. Notice that
$f_{a,b,c}^n(p_{n-m})=p_{n-m}$ and $(f_{a,b,c}^n)'(p_{n-m})=1$, we
obtain that $f_{a,b,c}^n(x)=x$ on $[p_{n-m}, c)$, which implies
that $c-$ is an $n$-periodic point. So $1$ is also an $n$-periodic point.

We denote the $n$-periodic orbit of $1$ as
$0<q_1<q_2<\cdots<q_{n-m-1}<q_{n-m}=c<q_{n-m+1}<\cdots<q_n=1$. Since
$f_{a,b,c}^n$ is linear on $[q_i, q_{i+1})$, it follows that
$f_{a,b,c}^n(x)=x$ on $[q_i, q_{i+1})$, $i=1,2,\ldots, n$. So
$f_{a,b,c}^n(x)=x$ on $I$.
\end{proof}

\subsection{Proof of Theorem A when $f_{a,b,c}(0)=f_{a,b,c}(1)$ and $ \log a / \log b $ is irrational} \
\ \ \ \ \ \ \ \ \

\smallskip
In this case,  according to Lemma \ref{theorem a}, we have only to
prove the following proposition.

\begin{proposition}\label{case2} Suppose $f_{a,b,c}(0)=f_{a,b,c}(1)$ and $ \log a / \log b $ is irrational,
there exists a constant $r>0$ such that $$rm(A)\le
m(f^{-n}_{a,b,c}(A)) \le m(A)/r,\ \ \ \ \ \forall A \in
\mathcal{B}.$$
\end{proposition}
\begin{proof}
The condition $f_{a,b,c}(0)=f_{a,b,c}(1)$ means
${\displaystyle c=\frac{1-b}{a-b}}$. Consider
$f_{a_1,b_1,c_1}:=f_{a,b,c}^{-1}$, the inverse map of $f_{a,b,c}$.
It is also a piecewise linear Lorenz map such that
$f_{a_1,b_1,c_1}(0)=f_{a_1,b_1,c_1}(1)$. In fact, we have
\begin{equation}\label{inverse}
a_1=\frac{1}{b},\ \ \ \ \ b_1=\frac{1}{a},\ \ \ \ \
c_1=\frac{(a-1)b}{a-b}.
\end{equation}
Let $\rho:=\rho(f_{a_1,b_1,c_1})$ be its rotation number. Write
 $$
 m^*_n(x)=\# \left\{0 \le i < n:f_{a_1,b_1,c_1}^i(x) \in (c_1,\ 1]
\right\} .
$$
By Lemma \ref{Rotation-Lyapunov},  we have
$a_1^{1-\rho}b_1^{\rho}=1$. Thus for all $ x \notin
C_{a_1,b_1,c_1}$, we have
\begin{eqnarray*}
(f^n_{a_1,b_1,c_1})'(x)&=& a_1^{n-m^*_n(x)}b_1^{m^*_n(x)}\\
& =& a_1^{n(1-\rho)}b_1^{n \rho}\cdot
\left(\frac{b_1}{a_1}\right)^{m^*_n(x)-n \rho}  =
\left(\frac{b_1}{a_1}\right)^{m^*_n(x)-n \rho}.
\end{eqnarray*}
According to  Lemma \ref{Rational-Rotation}, $|m^*_n(x)-n \rho|\le
4$. It follows that for all $ x \notin C_{a_1,b_1,c_1}$ and all $n
\ge 0$ we have
$$r \le (f^n_{a_1,b_1,c_1})'(x)\le 1/r, $$
where $r=\min \{b_1^4 a_1^{-4}, \ b_1^{-4} a_1^{4}\}=\min \{b^4
a^{-4}, \ b^{-4} a^{4}\}$. Consequently, by making a change of
variables we get
\begin{eqnarray*}
 r m(A) \le m(f_{a,b,c}^{-n}(A))
&=& \int {\bf 1}_A(f_{a,b,c}^n(x)) dx\\
&=& \int {\bf 1}_A(y) \cdot (f^{n}_{a_1,b_1,c_1})'(y) dy  \le
m(A)/r.
\end{eqnarray*}
\end{proof}

\subsection{Proof of Theorem A when
$f_{a,b,c}(0)<f_{a,b,c}(1)$} \ \ \

In this case, we are going to show

\begin{proposition}\label{case3}
Suppose $f_{a,b,c}(0)<f_{a,b,c}(1)$. There exists positive integer $n$ such that
$$(f_{a,b,c}^n)'(x)>\lambda>1$$ for all $x \in I$ except finite points.

\end{proposition}

\begin{proof}

The condition $f_{a,b,c}(0)<f_{a,b,c}(1)$ means $ac+(1-c)b>1$.
 So we must have $a>1$ or $b>1$. If we have both $a>1$ and
$b>1$. We take $n=1$. It remains to consider two cases: $a > 1 \ge b$ and $b
> 1 \ge a$.

First we assume  $a > 1 \ge b$. Let $f_{a_0,b,c}$ with
$a_0:=\frac{1-b(1-c)}{c}<a$, which is the homeomorphism defined by
(\ref{homeo1})(see Figure 2). We denote by $\rho_0$ the rotation
number of $f_{a_0,b,c}$. Notice that $0<\rho_0<1$.

For $x \in [0,\ 1]$, denote
$$
               m_n(x)=\sum_{i=0}^{n-1}{\bf 1}_{(c,\ 1]}(f_{a,b,c}^i(x)),
              \quad
\widetilde{m}_n(x)=\sum_{i=0}^{n-1}{\bf 1}_{(c,\
1]}(f_{a,b_0,c}^i(x)).
$$
According to Lemma \ref{rotation-frequence}, $\widetilde{m}_n(x)/n$
converges uniformly to $\rho_0$ as $n \to \infty$. So
$a_0^{1-\widetilde{m}_n(x)/n}b^{\widetilde{m}_n(x)/n}$ converges
uniformly to $1$ as $n \to \infty$ because
$a_0^{1-\rho_0}b^{\rho_0}=1$ (Lemma \ref{uniform}). Choose a
sufficiently small $\varepsilon_0>0$ such that
$$
(1-\varepsilon_0)\left(\frac{a}{a_0}\right)^{(1-\rho_0)/2}>1.
$$
There exists a positive integer $N_0\ge 1$ such that for all $x\in
[0,1]\backslash C_{a_0,b,c}$ and $ \forall n \ge N_0$, we have
$$
((f_{a_0,b,c}^n)'(x))^{\frac{1}{n}} = a_0^{1-\widetilde{m}_n(x)/n}
b^{\widetilde{m}_n(x)/n}>1-\varepsilon_0.
$$
Since $f_{a_0,b,c} \ge f_{a,b,c}(x)$, by Lemma
\ref{rotation-frequence} and  Lemma \ref{compare}, it is easy to see
that $m_n(x)\le \widetilde{m}_n(x)$ for $x \notin
C_*(n):=\bigcup_{i=0}^n(f^{-i}_{a,b,c}(c) \cup
f^{-i}_{a_0,b,c}(c))$. So, for large $n$ and $x \notin C_*(n)$, we
have
\begin{eqnarray} \label{expansive}
((f^n_{a,b,c})'(x))^{\frac{1}{n}}&=& a
\left(\frac{b}{a}\right)^{m_n(x)/n} \nonumber \\
&\ge&a \left(\frac{b}{a}\right)^{\widetilde{m}_n(x)/n} \nonumber \\
&=& ((f_{a_0,b,c}^n)'(x))^{\frac{1}{n}}
\left(\frac{a}{a_0}\right)^{1-\widetilde{m}_n(x)/n} \nonumber \\
& >& (1-\varepsilon_0)\left(\frac{a}{a_0}\right)^{(1-\rho_0)/2}
 > 1.
 \end{eqnarray}
This implies that for $n$ large we have $(f_{a,b,c}^{n})'(x)>1$ for
all $x \in I\backslash C_*(n)$. Obviously, $C_*(n)$ is consists of
finite points. So  $f_{a,b,c}^n$ is linear with slope greater than
$1$ on each component of $I\backslash C_*(n)$. It follows that
there exists positive integer $n$ and $\lambda>1$ such that
$(f_{a,b,c}^{n})'(x)>\lambda
>1$ for all $x \in I\backslash C_*(n)$.

The proof for  the case $b> 1\ge a$ is similar. We consider
$f_{a,b_0,c}$ with $b_0=\frac{1-ac}{1-c}<b$, which is the
homeomorphism defined by (\ref{homeo2}).
\end{proof}

\smallskip

\subsection{Proof of Theorem A when $f_{a,b,c}(0)>f_{a,b,c}(1)$}
\ \ \ \ \ \ \ \ \ \

We finish the proof of Theorem A by showing the following
proposition.

\begin{proposition}\label{case4}
Suppose $f_{a,b,c}(0)>f_{a,b,c}(1)$. There exists positive integer $n$ such that
$$(f_{a,b,c}^n)'(x)<\lambda<1$$ for all $x \in I$ except finite points.

\end{proposition}

\begin{proof}

The condition $f_{a,b,c}(0)>f_{a,b,c}(1)$ means $ac+(1-c)b<1$. So we
must have $a<1$ or $b<1$. If we have both $a<1$ and $b<1$, we take
$n=1$. It remains to consider two cases: $a < 1\le b$ and $a\ge
1>b$.

First we assume  $a <1 \le b$. Consider $f_{a,b_0,c}$ with
$b_0=\frac{1-ac}{1-c}>b$, which is the homeomorphism defined by
(\ref{homeo2})(see Figure 2). We denote by $\rho_1$ the rotation
number of $f_{a,b_0,c}$. Obviously,  $0<\rho_1<1$.


For $x \in [0,\ 1]$, put
$$
               m_n(x)=\sum_{i=0}^{n-1}{\bf 1}_{(c,\ 1]}(f_{a,b,c}^i(x)),
              \quad
\widetilde{m}_n(x)=\sum_{i=0}^{n-1}{\bf 1}_{(c,\
1]}(f_{a,b_0,c}^i(x)).
$$
According to Lemma \ref{rotation-frequence}, $\widetilde{m}_n(x)/n$
converges uniformly to $\rho_0$ as $n \to \infty$. So
$a^{1-\widetilde{m}_n(x)/n}b_0^{\widetilde{m}_n(x)/n}$ converges
uniformly to $1$ as $n \to \infty$ because
$a^{1-\rho_1}b_0^{\rho_1}=1$ (Lemma \ref{uniform}). Choose a
sufficiently small $\varepsilon_1>0$ such that
$$
(1+\varepsilon_1)\left(\frac{b}{b_0}\right)^{\rho_1/2}<1.
$$
There exists a positive integer $N_0\ge 1$ such that for all $x\in
[0,1]\backslash C_{a,b_0,c}$ and $ \forall n \ge N_0$, we have
$$
((f_{a,b_0,c}^n)'(x))^{\frac{1}{n}} =
a^{1-\frac{\widetilde{m}_n(x)}{n}}
b_0^{\frac{\widetilde{m}_n(x)}{n}}<1+\varepsilon_1.
$$
Since $f_{a,b,c} \le f_{a,b_0,c}(x)$, by Lemma
\ref{rotation-frequence} and  Lemma \ref{compare}, it is easy to see
that $m_n(x)\le \widetilde{m}_n(x)$ for $x \notin
C^*(n):=\bigcup_{i=0}^n(f^{-i}_{a,b,c}(c) \cup
f^{-i}_{a,b_0,c}(c))$. So, for large $n$ and $x \notin C^*(n)$, we
have
\begin{eqnarray} \label{expansive}
((f^n_{a,b,c})'(x))^{\frac{1}{n}}&=& a
\left(\frac{b}{a}\right)^{m_n(x)/n} \nonumber \\
&\le&a \left(\frac{b}{a}\right)^{\widetilde{m}_n(x)/n} \nonumber \\
&=& ((f_{a,b_0,c}^n)'(x))^{\frac{1}{n}}
\left(\frac{b}{b_0}\right)^{\widetilde{m}_n(x)/n} \nonumber \\
& <& (1+\varepsilon_1)\left(\frac{b}{b_0}\right)^{\rho_1/2}
 < 1.
 \end{eqnarray}

This implies that for $n$ large we have $(f_{a,b,c}^{n})'(x)<1$ for
all $x \in I\backslash C^*(n)$. Since $f_{a,b,c}^n$ is a piecewise
contract linear map with at most finite pieces, there exists
$\lambda<1$ such that  $(f_{a,b,c}^{n})'(x)<\lambda <1$ for all $x
\in I$ except at most finite points.

The proof for  the case $a \ge 1 > b$ is similar. We consider
$f_{a_0,b,c}$ with $a_0=\frac{1-(1-c)b}{c}>a$, which is the
homeomorphism defined by (\ref{homeo1}).
\end{proof}

\smallskip

\smallskip

\subsection{Diffeomorphic conjugacy.}

 A partial result of Theorem A may be obtained in a
different way. If
\begin{equation}\label{condition}
c\sqrt{a}+(1-c)\sqrt{b}>1,
\end{equation}
which is stronger than $f_{a,b,c}(0) \ge f_{a,b,c}(1)$,
 it is possible to find some diffeomorphism
$h$ such that $h\circ f_{a,b,c}\circ h^{-1}$ is piecewise
expanding. We do find such a  diffeomorphism among the
one-parameter group of transformations $h_s:[0,1] \to [0,1]$
($s\in \mathbb{R}_+$) defined by
$$h_s(x)=\frac{sx}{1+(s-1)x}.$$

We can also prove that the above condition (\ref{condition}) is
actually necessary for the existence of such a diffeomorphism $h_s$.
This was one starting point of our study on acim of piecewise linear
Lorenz map.

\setcounter{equation}{0}

\section{Equivalence}

Let $f:=f_{a,b,c}$ be a piecewise linear Lorenz map with
$ac+b(1-c)>1$. The acim $\mu$ is not necessarily equivalent to the
Lebesgue measure $m$, even if $f$ is strongly expanding. Parry
\cite{P3} proved that the acim of $f_{a,a,1/2}$ is not equivalent to
the Lebesgue measure if and only if $1<a<\sqrt{2}$. We first show
that the equivalence between its acim and the Lebesgue measure  is
nothing but the transitivity of $f$.

\subsection{Equivalence and transitivity}
\begin{lemma}\label{transitive} Let $f:=f_{a,b,c}$ be a piecewise linear Lorenz map with $ac+b(1-c)>1$.
Then the acim of $f$ is equivalent to the Lebesgue measure if and
only if $f$ is transitive.\end{lemma}

\begin{proof} Let $h$ be the density of the  acim $\mu$, i.e.
\begin{equation}\label{acim} \mu(A)=\int_A h(x) d x,\ \ \ \ \forall A \in
\mathscr{B}, \end{equation} and let $\mbox{\rm supp}\ (\mu)$ be the
support of $\mu$. Obviously, $\mbox{\rm supp} (\mu)$ is an invariant
closed set of $f$.

The measure $\mu$ is equivalent to the Lebesgue measure $m$ if and
only if $\mbox{\rm supp} (\mu)=I$. In fact, $\mbox{\rm supp}
(\mu)=I$ means $h(x)>0$ for $m$-a.e. $x \in I$. By (\ref{acim}),
$\mu(A)=0$ implies $m(A)=0$.

Now we show that $\mbox{\rm supp}(\mu)=I$ if and only if $f$ is
transitive. At first, it is easy to see the non transitivity of $f$
implies $I\backslash \mbox{\rm supp}\ (\mu)$ is nonempty. On the
other hand, notice that $\mbox{\rm supp} (\mu)$ contains some
interval $J$ because  $h$ is of bounded variation. So, the
transitivity of $f$ implies
$$
I=\overline{\bigcup_{n=0}^{\infty}f^n(J)}\subseteq \mbox{\rm supp}\
(\mu).
$$
\end{proof}

Now we are going to discuss the transitivity of piecewise linear
Lorenz maps by using the renormalization theory  of expanding Lorenz
map (cf. \cite{Dy, GS}).

\subsection{Proof of Theorem B} \ \ \

Let $f:=f_{a,b,c}$ be a piecewise linear Lorenz map with
$ac+b(1-c)>1$, $\kappa$ be the minimal period.

If $\kappa=1$, then $f$ is not renormalizable (\cite{Dy}), which
implies that $f$ is transitive.

In what follows we assume $\kappa >1$. Let $O$ be the unique
$\kappa$-periodic orbit, and $P_L$ and $P_R$ be adjacent
$\kappa-$periodic points so that $[P_L,\ P_R]$ contains the critical
point $c$.  By Lemma \ref{minimal period}, $f^{\kappa}$ is
continuous and linear on $[P_L,\ c)$ and on $(c,\ P_R]$. Put
$$A:=f^{\kappa}(c+), \ \ \ \ \ \ B:=f^{\kappa}(c-).$$
We discuss the transitivity of $f$ by distinguish the following
three cases:
\begin{enumerate}
\item $[A,\ B] \backslash [P_L,\ P_R] \neq \emptyset$;

\item $[A,\ B]=[P_L,\ P_R]$;

\item $[A,\ B] \subsetneqq [P_L,\ P_R]$.

\end{enumerate}


{\em Case (1)}. In this case, we have  $D:=\overline{\bigcup_{n\ge
0}f^{-n}(O)} \neq O$. It follows from Lemma \ref{dic} that $f$ is
not renormalizable. So $f$ is transitive.

If $[A,\ B] \subseteq [P_L,\ P_R]$, $f$ admits a renormalization
\[
Rf(x)=\left \{
\begin{array}{ll}
f^{\kappa}(x) & x \in [A,\  c) \\
f^{\kappa}(x) & x \in (c,\  B].
\end{array}
\right.
\]

{\em Case (2)}. In this case, $Rf$ is not renormalizable because it
admits fixed point. So $Rf$ is a transitive Lorenz map on $[P_L,\
P_R]$. By equation (\ref{equality}) in Lemma \ref{minimal period},
$f$ is also transitive.

{\em Case (3)}. In this case, we have $P_L<A$ or $B<P_R$. Assume
 that $P_L<A$. Since $[A,\ B] \subseteq [P_L,\ P_R]$, it follows  $(P_L,\ A) \bigcap \ (\bigcup_{n \ge 0}
f^{n}([A,\ B]))=\emptyset$, which indicates $f$ is not transitive.
Similarly,  $B<P_R$ implies $f$ is not transitive.

Now we consider the special case $\kappa=2$. According to Lemma
\ref{minimal period}, $\kappa=2$ implies $f$ admits no fixed point,
$A=f(0)=1-ac<b(1-c)=f(1)=B$ and $c \in [A,\ B]$.

A simple computation shows
\[ f^2(x)=\left \{ \begin{array}{lcl}
a^2x+(1-ac)(a+1) & \mbox{\rm if} &  x \in [0,  \frac{ac+c-1}{a}], \\
abx-abc+b-bc & \mbox{\rm if} &  x \in (\frac{ac+c-1}{a}, c ], \\
abx-abc-ac+1& \mbox{\rm if} &  x \in (c,\frac{bc+c}{b}],\\
b^2x-b^2c-bc & \mbox{\rm if} &  x \in (\frac{bc+c}{b},1 ].
\end{array}
\right. \label{map2}
\]
It follows that the map $f$ admits two $2$-periodic points:
\[
P_L=\frac{abc+bc-b}{ab-1}, \ \ \ \ \ \ \ \ \ \
P_R=\frac{abc+ac-1}{ab-1}\label{periodic pts}.
\]

So,  $[A,\ B]\subseteq [P_L,\ P_R]$ is equivalent to
$$
A \geq \frac{abc+bc-b}{ab-1} \ \ \ \ \ \ \ \ \mbox{\rm and}   \ \ \
\ \ \ \ B \leq \frac{abc+ac-1}{ab-1}
$$
or equivalently
$$ab \leq 1+\frac{B-c}{c-A}\ \ \ \ \ \ \ \mbox{\rm and} \ \ \ \ \ \ \
 \ \ \ \ \ ab \leq 1+\frac{c-A}{B-c}. $$
 Recall $M:=\min \left\{\frac{B-c}{c-A},\
\frac{c-A}{B-c}\right\}$. We have $[A,\ B] \subsetneqq [P_L,\ P_R]$
if and only if
  \[\left \{ \begin{array}{lcl}
1<ab\leq 1+M &  \mbox{\rm if} & M<1 \\
1<ab \le 2 &  \mbox{\rm if} &  M=1.
\end{array}
\right.
\]
In other words, $f$ is transitive if and only if
 \[\left \{ \begin{array}{lcl}
ab>1+M & \mbox{\rm if} & M<1 \\
ab\ge 2 & \mbox{\rm if} &  M=1.
\end{array}
\right.
\]

\vspace{0.5cm}

\setcounter{equation}{0}

\section{The densities of the acims}

We finish the paper by pointing out how to obtain the density of the
acim in some special cases.

\subsection{$\beta$-transformation} \ \ \

The first case is the special Lorenz maps $f_{a,a,c}$ $(a>1)$. It
was known that they admit their acims (see also Theorem A). Gelfond
\cite{Gel} and  Parry \cite{P1,P} had determined the density of the
acim of $f_{a,a,c}$, which is up to a multiplicative constant equal
to
\begin{equation}
g(x)=\sum_{f_{a,a,c}^n(0)<x} \frac{1}{a^n}-\sum_{f_{a,a,c}^n(1)>x}
\frac{1}{a^n}. \label{density}
\end{equation}

Suppose that the acim of  $f_{a,b,c}$ exists but is not equivalent
to the Lebesgue measure.
 From the proof of Theorem B  (see
Section 3) we have seen that the restriction of $f^{\kappa}_{a,b,c}$
on $[f_{a,b,c}^{\kappa}(c+),f^{\kappa}_{a,b,c}(c-)]$, where $\kappa$
is the minimal period of periodic points of $f_{a,b,c}$, is a
piecewise Linear Lorenz map of the form $f_{a_*,a_*,c_*}$ on the
renormalization interval
$[f_{a,b,c}^{\kappa}(c+),f^{\kappa}_{a,b,c}(c-)]$. Thus, we can
obtain the density $g_*(x)$ of the acim of $f_{a_*,a_*,c_*}$ by
using (\ref{density}). Then we get the density of the acim of
$f_{a,b,c}$:
   \begin{equation}
  g_{a,b,c}(x)=\frac{1}{\kappa}\left[g_*(x)+P_{a,b,c}g_*(x)+\cdots+P_{a,b,c}^{\kappa-1}g_*(x)\right],
   \end{equation}
where $P_{a,b,c}$ is the Frobenius-Perron operator associated to
$f_{a,b,c}$. Actually we can easily check that $$P_{a,b,c}^{\kappa}
g_*(x)=g_*(x).$$


\subsection{Piecewise linear Markov map}\ \ \

The second case is $f_{\beta,\alpha,1/\beta}\ (\beta>1, 0<\alpha\le
\frac{\beta}{\beta-1})$, which is the piecewise linear Lorenz map
$S_{\beta, \alpha}$ studied by Dajani et al in \cite{DHK}. Remember
that $S_{\beta, \alpha}$ is defined by equation
(\ref{LinearLorenzMap10}), and we only assume $\beta>1$ rather than
$1<\beta<2$ in \cite{DHK}. If
$\alpha=\frac{1}{\beta^{k-1}(\beta-1)}$ for some integer $k\ge 1$,
then $f_{\beta,\alpha,1/\beta}$ is a piecewise linear Markov map.

\begin{proposition} \label{non-linear} Assume $\beta>1$ and
$\alpha=\frac{1}{\beta^{k-1}(\beta-1)}$ for some integer $k\ge 1$.
Then the density of the acim of $f_{\beta,\alpha,1/\beta}$ is up to
a multiplicative constant equal to
$$
g_{\beta,k}(x)=\frac{1}{\beta-1} {\bf 1}_{[0,\ 1]}(x)+ \sum_{i=1}^k
\beta^{i-1}{\bf 1}_{[0,\ \beta^{-i}]}(x).
$$
\end{proposition}
\proof Let  $\cC$ be the partition of $[0,1]$ given by
$0<\beta^{-k}<\beta^{-(k-1)}< \ldots <\beta^{-2} <\beta^{-1}<1.$ One
can easily check that $f_{\beta,\alpha,1/\beta}$ is a piecewise
linear Markov map with respect to the partition $\cC$. Let $P$ be
the Perron-Frobenius operator of $f_{\beta,\alpha,1/\beta}$ (see
Section 2.3).

Note that
\begin{eqnarray*}
P 1 (x) &=& \frac{1}{\beta} {\bf 1}_{[0,\ 1]}(x) +(\beta-1)\beta^{k-1} {\bf 1}_{[0,\ \beta^{-k}]}(x) \\
P {\bf 1}_{[0,\ \beta^{-k}]}(x)&=&\frac{1}{\beta} {\bf 1}_{[0,\ \beta^{-(k-1)}]}(x)\\
&\vdots&\\
P {\bf 1}_{[0,\ \beta^{-2}]}(x)&=& \frac{1}{\beta} {\bf 1}_{[0,\ \beta^{-1}]}(x) \\
P {\bf 1}_{[0,\ \beta^{-1}]}(x) &=& \frac{1}{\beta} {\bf 1}_{[0,\
1]}(x).
\end{eqnarray*}
We obtain
\begin{eqnarray*}
P g_{\beta,k}(x) &=& \frac{1}{\beta-1} P {\bf 1}_{[0,\ 1]}(x) +
\sum_{i=1}^k \beta^{i-1} P{\bf 1}_{[0,\ \beta^{-i}]}(x) \\
&=& \frac{1}{\beta(\beta-1)} {\bf 1}_{[0,\ 1]}(x) + \beta^{k-1} {\bf
1}_{[0,\ \beta^{-k}]}(x) +
\sum_{i=2}^k \beta^{i-2}{\bf 1}_{[0,\ \beta^{-(i-1)}]}(x) + \frac{1}{\beta}{\bf 1}_{[0,\ 1]}(x) \\
&=& \frac{1}{\beta-1} {\bf 1}_{[0,\ 1]}(x) +\sum_{i=1}^k \beta^{i-1} {\bf 1}_{[0,\ \beta^{-i}]}(x) \\
&=&g_{\beta,k}(x).
\end{eqnarray*}
$\hfill \Box$

The special map of the form $f_{1,n,1-1/n}$ ($n \ge 2$ being an
integer) is also a piecewise linear Markov map. It was proved in
\cite{DF2} that its density is equal to
$$
g_{n}(x)=\frac{2}{n+1}\sum^{n-1}_{i=0}{\bf 1}_{[ \frac{i}{n},\
1]}(x).
$$


\vspace{0.1cm} \noindent DING Yi Ming: Wuhan Institute of Physics
and Mathematics, The Chinese Academy of Sciences, Wuhan  430071,
P.R.China \\
\noindent {\it Email:\ } ding@wipm.ac.cn \\

\vspace{0.1cm} \noindent FAN Ai Hua: Department of Mathematics,
LAMFA, UMR 6140 CNRS , University of Picardie, 33 Rue Saint Leu,
80039 Amiens Cedex 1, France \\
 \noindent {\it Email:\ } ai-hua.fan@u-picardie.fr\\

\vspace{0.1cm} \noindent YU Jing Hu: Department of Mathematics,
Wuhan University of Technology, Wuhan 430070, P.R.China \\
 \noindent {\it Email:\ } yujh67@126.com\\
\end{document}